\documentclass[3p,10pt,a4paper,twoside,fleqn,sort&compress]{filomat}
\usepackage{amssymb,amsmath,latexsym}
\usepackage[varg]{pxfonts}
\newtheorem{theorem}{Theorem}[section]

\newcommand{\FilVol}{xx}
\newcommand{\FilYear}{yyyy}
\newcommand{\FilPages}{zzz--zzz}
\newcommand{\FilDOI}{(will be added later)}
\newcommand{\FilReceived}{Received: dd Month yyyy; Accepted: dd Month yyyy}
\begin{document}

%%%%%% TO BE ENTERED BY THE AUTHOR(S)
%%%
%%% ENTER TITLE
\title{Summation of $p$-Adic Functional Series in Integer Points}

%%% AUTHOR(S) FULL NAMES, AND EMAIL ADDRESSES
\author[affil1]{Branko Dragovich}
\ead{dragovich@ipb.ac.rs}
%%%
\author[affil2]{Andrei Yu. Khrennikov}
\ead{Andrei.Khrennikov@lnu.se}
%%%
\author[affil3]{Nata\v sa \v Z. Mi\v si\'c}
\ead{nmisic@afrodita.rcub.bg.ac.rs}
%%% ENTER AUTHOR(S) AFFILIATION(S)
\address[affil1]{Institute of Physics, University of Belgrade,  Belgrade, Serbia, and \\ Mathematical Institute of the Serbian Academy of Sciences and Arts, Belgrade, Serbia}
\address[affil2]{International Center for Mathematical Modeling
in Physics, Engineering, Economics, and Cognitive Science\\
Linnaeus University, V\"axj\"o-Kalmar, Sweden}
\address[affil3]{Lola Institute, Kneza Vi\v seslava 70a, Belgrade, Serbia}
%%% AND CORRESPONDINGLY FOR OTHER AUTHORS, IF THERE ARE MORE AUTHORS
%%% ENTER ABBREVIATED AUTHOR(S) NAMES FOR PAGE HEADINGS
%\newcommand{\AuthorNames}{F. Author, S. Author}
%%% IF THERE ARE MORE THAN TWO AUTHORS WRITE
 \newcommand{\AuthorNames}{B. Dragovich et al.}
%%%

%%% ENTER MSC, KEYWORDS, RECEIVED, EDITOR, THANKS FOR FINANCIAL SUPPORT FOR RESEARCH
\newcommand{\FilMSC}{Primary 11E95, 40A05; Secondary 11D88, 11B83}
\newcommand{\FilKeywords}{($p$-adic series, $p$-adic invariant summation, integer sequences, Bell numbers)}
\newcommand{\FilCommunicated}{(name of the Editor, mandatory)}
%\newcommand{\FilSupport}{Research supported by ... (optionally)}
%%% If you do not want to thank for the financial support of the research, remove
%%% the previous line (i.e., leave \FilSupport undefined)
%%%%%%%%%%%%%%%%%%%%%%%%%%%%%%%%%%%%%%%%%%%%%%%%%%

\begin{abstract}
Summation of a large class of the functional series, which terms contain factorials, is considered.
We first investigated finite partial sums for integer arguments. These sums have the same values in
 real and all $p$-adic cases. The corresponding infinite  functional series are divergent in the real case,
 but they are convergent and have $p$-adic invariant sums in $p$-adic cases. We found polynomials which generate all
significant ingredients of these series and make connection between their real and $p$-adic properties. In particular,
we found connection of one of our integer sequences with the Bell numbers.
\end{abstract}

\maketitle

%%%%%% THIS PART MUST BE PLACED IMMEDIATELY AFTER THE \maketitle COMMAND
%%%%%% BACK TO ORIGINAL FOOTNOTES
\makeatletter
\renewcommand\@makefnmark%
{\mbox{\textsuperscript{\normalfont\@thefnmark)}}}
\makeatother
%%%%%%

\section{Introduction}

The infinite series play an important role in mathematics, physics and many other applications. Usually their numerical
ingredients  are rational numbers and therefore the series can be treated in any  $p$-adic as well as in real
number field, because rational numbers are endowed by real and $p$-adic norms simultaneously. Hence, for
a real divergent series  it may be useful  investigation of its $p$-adic analog  when  $p$-adic sum is a rational number for a rational argument.

%%%%%%%%%%%%%%%%%%%%%%%%%%%%%%%%%%%%%%%%%%

Many series in string theory, quantum field theory, classical and quantum mechanics contain factorials. Such series are usually divergent in
the real case and convergent in $p$-adic ones. This was main motivation for considering different $p$-adic aspects of the series with factorials
 in \cite{bd1,bd2,bd3,bd4,bd5,bd6,bd7,bd8,bd9,bd10,bd11} and many summations  performed  in rational points. Also, using $p$-adic number
field invariant summation in rational points, rational summation \cite{bd5} and adelic summation \cite{bd2} were introduced.

It is worth mentioning that $p$-adic numbers and $p$-adic analysis have been successfully applied in modern mathematical physics (from strings to complex systems
and the universe as a whole) and in some related fields (in particular in bioinformation systems, see, e.g. \cite{bd13}), see \cite{freund,vvz} for an early review and \cite{bd12} for a recent one.  Quantum models with $p$-adic valued wave functions, see, e.g.  \cite{AK1} for
the recent review, generated various $p$-adic series leading to nontrivial summation problems (see, e.g. \cite{AK2,AK3,AK4}).

In this paper we consider $p$-adic invariant summation of a wide class of finite and infinite functional series which terms contain factorials, i.e.
$\sum \varepsilon^n (n + \nu)! P_{k\alpha} (n; x) x^{\alpha n + \beta} ,$ where $\varepsilon = \pm 1, $ and parameters $  \nu, \beta \in \mathbb{N}_0 = \mathbb{N}\bigcup \{ 0\} , \,\,
k,\alpha \in \mathbb{N} .$ $ \,\, P_{k\alpha} (n; x)$ are polynomials in $x$ of degree $k\alpha$  which coefficients are some polynomials in $n$.
 We  show that there exist polynomials $P_{k\alpha} (n; x)$ for any degree $k\alpha,$ such that for any $x \in \mathbb{Z}$  values of the sums do not depend on $p .$  Moreover, we have found
recurrence relations  to calculate such  $P_{k\alpha} (n; x)$ and other relevant polynomials.  The obtained results are generalization of recently
obtained ones for the series $\sum n! P_k (n; x) x^{n} ,$ see \cite{bd14}.  Some results are illustrated by simple examples.

 All necessary general information on $p$-adic series can be found in standard books on $p$-adic analysis, see, e.g. \cite{schikhof}.

%\newpage

%\section{The second section}

\section{Some Functional Series with Factorials}

We consider  functional series of the form \begin{align}
S_{k\alpha} (x) = \sum_{n = 0}^{+\infty} \varepsilon^n \, (n + \nu)! \, P_{k\alpha} (n; x) \, x^{\alpha n + \beta} \,, \quad \varepsilon = \pm 1\,, \,\,  \nu, \beta \in \mathbb{N}_0 = \mathbb{N} \cup \{ 0\} \,, \,\,
\alpha,\, k \in \mathbb{N}\,, \label{2.1}
\end{align}
where
\begin{align} \label{2.2}
& P_{k\alpha} (n;x) = C_{k\alpha} (n)\, x^{k\alpha} + C_{(k-1)\alpha} (n)\, x^{(k-1)\alpha} + \dots + C_\alpha(n)\, x^\alpha + C_0(n)\,,   \nonumber \\
& C_{i\alpha}(n) = \sum_{j=0}^i c_{ij}\, n^{j\alpha} \,, \,\, \, 0 \leq i \leq k \,, \, \,  c_{ij} \in \mathbb{Z} .
\end{align}

Since rational numbers belong to real as well as to $p$-adic numbers, the series \eqref{2.1} can be considered as real ($x \in \mathbb{R}$) as $p$-adic
($x \in \mathbb{Q}_p$) ones. In the real case, \eqref{2.1} is evidently divergent. In the sequel we shall investigate \eqref{2.1} $p$-adically.

\subsection{Convergence of the $p$-Adic Series}

Necessary and sufficient condition for the $p$-adic power series
to be convergent \cite{schikhof,vvz} coincides, i.e. 
\begin{align}
S(x) = \sum_{n=1}^{+\infty} a_n x^n ,   \quad a_n \in \mathbb{Q} \subset \mathbb{Q}_p ,
\quad x \in \mathbb{Q}_p ,  \quad |a_n x^n|_p \to 0 \,\, \text{as} \,\, n
\to \infty ,     \label{2.3}
\end{align}
where $|\cdot|_p$ denotes $p$-adic absolute value (also called $p$-adic norm). To prove this assertion,
 note that $p$-adic absolute value is ultrametric (non-Archimedean) one and satisfies  inequality
$ |x + y|_p \leq \text{max} \{|x|_p, |y|_p\}.$  Now suppose that the series \eqref{2.3} is convergent for some arguments $x$ and the corresponding sum is $S(x) ,$ i.e $|S(x) - S_n (x)|_p \to 0 \,\, \text{as} \,\, n \to \infty ,$ where $S_n (x) = a_0 + a_1 x + ...+ a_{n-1} x^{n-1} .$ Then $|a_n x^n|_p = |S_{n+1}(x) - S_n (x)|_p = |S_{n+1}(x) - S(x)+ S(x) - S_n (x)|_p \leq \text{max} \{|S(x) - S_{n+1}(x)|_p \,, |S(x) - S_{n}(x)|_p\} \to 0 \,\, \text{as} \,\, n \to \infty .$
That $|a_n x^n|_p \to 0 \,\, \text{as} \,\, n \to \infty $ is sufficient condition  follows from the Cauchy criterion. Namely, for enough large $n$ and arbitrary $m,$ due to ultrametricity one can write  $|a_n x^n|_p  =|a_n x^n + a_{n+1} x^{n+1} + \cdots + a_{n+m} x^{n+m}|_p \,.$

 The functional series \eqref{2.1} contains $(n+\nu)!$, hence to investigate its convergence one has to know $p$-adic norm of $(n+\nu)! \,.$
 First, one has to know a power $M(n)$ by which  prime $p$ is contained in $n!$  (see, e.g. \cite{vvz} or \cite{bd14}). Let $n = n_0 + n_1 p +
... + n_r p^r $ and $ s_n = n_0 + n_1 + ... + n_r $ denotes the sum of digits in expansion of a natural number $n$ in base $p$.
 Then, one has
 \begin{align} &n! = m\, p^{M(n)} = m\,  p^{\frac{n -s_n}{p-1}} ,\quad p \nmid  m ,  \quad  |n!|_p = p^{-\frac{n -s_n}{p-1}} \,, \nonumber  \\
 &|(n+\nu)!|_p = p^{-\frac{n+\nu -s_{n+\nu}}{p-1}} . \label{2.4} \end{align}

\begin{theorem}
$p$-Adic series \eqref{2.1} is convergent for every $x \in \mathbb{Z}_p$ and any $p .$
\end{theorem}

\begin{proof}
Consider $p$-adic norm of the general term in \eqref{2.1} when $x \in \mathbb{Z}_p \,,$ i.e.
\begin{align}  |\varepsilon^n \, (n+\nu)! \, P_{k\alpha}(n; x) \,  x^{\alpha n +\beta}|_p \leq  |(n +\nu)!|_p = p^{- \frac{n +\nu - s_{n +\nu}}{p-1}} \to 0
\,\, \text{as} \,\, n \to \infty \,,\label{2.5} \end{align}
where $|P_{k\alpha}(n; x)|_p \leq 1$ and $|x^{\alpha n +\beta}|_p \leq 1 \,.$
Hence,  the power series $\sum_{n=0}^\infty \varepsilon^n \, (n+\nu)! \, P_{k\alpha}(n; x) \,  x^{\alpha n +\beta}$
is convergent in $\mathbb{Z}_p$, i.e. $|x|_p \leq 1 \,.$
\end{proof}

 Since $\bigcap_p \mathbb{Z}_p = \mathbb{Z}$, it means that the
infinite series $\sum_{n=0}^\infty \varepsilon^n \, (n+\nu)! \, P_{k\alpha}(n; x) \,  x^{\alpha n +\beta}$ is simultaneously convergent for all
integers and all $p$-adic norms.

\section{Summation at Integer Points}

 Mainly we are interested for which polynomials $P_{k\alpha} (n; x)$ we have that if $x \in \mathbb{Z}$ then the sum of the series
(1) is $S_{k\alpha} (x) \in \mathbb{Z}$, i.e. $S_{k\alpha} (x)$ is also an integer which is the same in all $p$-adic cases. Since polynomials $P_{k\alpha} (n; x)$ are determined
by polynomials  $C_{i\alpha} (n), \, 0 \leq i \leq k $ \eqref{2.2},  it means that one has to find these $C_{i\alpha} (n), \, 0 \leq i \leq k .$  Our task is to find connection
between polynomial $P_{k\alpha} (n; x)$ and sum of infinite series $S_{k\alpha} (x),$ which  becomes also a polynomial.

We are interested now in determination of the polynomials $P_{k\alpha} (n; x)$ and the corresponding sums $S_{k\alpha} (x) = Q_{k\alpha} (x)$ of the infinite series
\eqref{2.1}, where
\begin{align}
Q_{k\alpha} (x)= q_{k\alpha} \, x^{k\alpha} \, + \, q_{(k-1)\alpha} \, x^{(k-1)\alpha} \, + \cdots + \, q_\alpha \, x^\alpha \, + \, q_0   \label{3.1}
\end{align}
are also some polynomials related to $P_{(k\alpha)} (n; x) ,$  so that $P_{(k\alpha)} (n; x) $ and $Q_{(k\alpha)} (x)$ do not depend on concrete
$p$-adic consideration and that they are valid for all $x \in \mathbb{Z} .$

A very simple and illustrative example  \cite{schikhof} of $p$-adic invariant summation of the infinite series \eqref{2.1} is
\begin{align} \sum_{n \geq 0} n! \, n = 1! 1 + 2! 2 + 3! 3 + ... = - 1 \label{3.2} \end{align}
which obtains  taking  $x = 1, \, \, P_{11}(n; 1) = n$ and gives $Q_{11} (1) = - 1.$
 To prove  \eqref{3.2}, one can use any one of the following two properties:
\begin{align} \sum_{n = 1}^{N-1} n! \, n  = -1 + N!  \,, \quad \, \qquad  n! n = (n+1)! - n! \,. \label{3.3} \end{align}

In the sequel we shall develop and apply approach of partial sums which generalize the first one in \eqref{3.3}.

\subsection{The Partial Sums}

Having in mind our goal on rational summation of the functional series \eqref{2.1}, let us consider
the  partial sums of its simplified version. Namely,
\begin{align}
S_k (N; x) &= \sum_{n=0}^{N-1} \varepsilon^n \, (n+\nu)! \, (n +\nu)^k \, x^{\alpha n +\beta} = \nu! \, \nu^k \, x^\beta +
\sum_{n=1}^{N-1} \varepsilon^n \, (n+\nu)! \, (n +\nu)^k \, x^{\alpha n +\beta} \nonumber \\
&= \nu! \, \nu^k \, x^\beta + \varepsilon \, x^\alpha \,\sum_{n=0}^{N-1} \varepsilon^{n} \, (n+\nu)! \, (n +\nu +1)^{k+1} \, x^{\alpha n +\beta}
- \varepsilon^N \, (N+\nu)! \, (N+\nu)^k \, x^{\alpha N +\beta} \nonumber \\
&= \nu! \, \nu^k \, x^\beta + \varepsilon \, x^\alpha \,\sum_{n=0}^{N-1} \varepsilon^{n} \, (n+\nu)! \, \sum_{\ell =0}^{k+1} \left(\begin{aligned} k&+1 \\
&\ell \end{aligned}\right)\, (n +\nu)^{\ell} \, x^{\alpha n +\beta} - \varepsilon^N \, (N+\nu)! \, (N+\nu)^k \, x^{\alpha N +\beta}  \nonumber     \\
&= \nu! \, \nu^k \, x^\beta + \varepsilon \, x^\alpha  \,  S_{0} (N; x) + \varepsilon \, x^\alpha  \, \sum_{\ell =1}^{k+1} \left(\begin{aligned} k&+1 \\
&\ell \end{aligned}\right)\, S_{\ell} (N; x) - \varepsilon^N \, (N+\nu)! \, (N+\nu)^k \, x^{\alpha N +\beta} \,,  \label{3.4}
\end{align}
where $S_0 (N; x)  = \sum_{n=0}^{N-1} \varepsilon^n \, (n +\nu)! \, x^{\alpha n + \beta}$. Obtained recurrence relation \eqref{3.4} gives possibility to
find sums $S_k (N; x) \,, \, \, k \in \mathbb{N}\,,$ with respect to $S_0 (N; x) .$ Performing operations for $k =0$ and $k = 1$ in \eqref{3.4}, one obtains
\begin{align}
S_1 (N; x) = \, &(\varepsilon \, x^{-\alpha} - 1) \, S_0(N; x) - \varepsilon \, \nu! \, x^{\beta - \alpha}  + \varepsilon^{n-1} \, (N+\nu)!\, x^{\alpha N + \beta - \alpha} \,,  \label{3.5}   \\
S_2 (N; x) = \, &((\varepsilon \, x^{-\alpha} - 2)(\varepsilon \, x^{-\alpha} - 1) -1)\, S_0(N; x) + \varepsilon \, \nu! \, x^{\beta - \alpha}\, (2 -\varepsilon \, x ^{-\alpha} - \nu) \nonumber \\   & \, + (\varepsilon \, x^{-\alpha} -2 + N + \nu)\, \varepsilon^{n-1} \, (N+\nu)!\, x^{\alpha N + \beta - \alpha} \,. \label{3.6}
\end{align}
Equations \eqref{3.5} and \eqref{3.6} can be rewritten in equivalent and more suitable form, respectively:
\begin{align}
&\sum_{n=0}^{N-1} \varepsilon^n \, (n + \nu)! [x^\alpha\, (n+\nu) + x^\alpha - \varepsilon] \, x^{\alpha n +\beta}  = - \varepsilon \, \nu! \, x^\beta + \varepsilon^{N-1} \, (N + \nu)!\, x^{\alpha N + \beta} \,, \label{3.7} \\
&\sum_{n=0}^{N-1} \varepsilon^n \, (n + \nu)! [ x^{2\alpha}\, (n+\nu)^2 - ( x^{2\alpha} - 3\, \varepsilon \, x^\alpha +1) ] \, x^{\alpha n +\beta}  =
\varepsilon \, \nu! \, [(2-\nu) x^\alpha - \varepsilon] \, x^\beta \nonumber \\
& + [(N + \nu - 2)\, x^\alpha + \varepsilon] \, \varepsilon^{N-1} \, (N + \nu)!\, x^{\alpha N + \beta} \,. \label{3.8}
\end{align}

\begin{theorem}
\label{Th1}
The recurrence relation \eqref{3.4} has solution in the form
\begin{align}
\sum_{n=0}^{N-1} \varepsilon^n \, (n + \nu)!\, [(n+\nu)^k \, x^{k \alpha} \, + U_{k\alpha} (x)] \, x^{\alpha n +\beta}  = V_{(k-1)\alpha} (x) + A_{(k-1)\alpha} (N; x) \,\varepsilon^{N-1} \, (N + \nu)!\, x^{\alpha N + \beta} \,, \label{3.9}
\end{align}
where polynomials $U_{k\alpha} (x) \,, V_{(k-1)\alpha} (x)$ and $A_{(k-1)\alpha} (N; x)$ satisfy the following recurrence relations:
\begin{align}
&\sum_{\ell =1}^{k+1} \left(\begin{aligned}k&+1\\ &\ell\end{aligned}\right) \, x^{(k-\ell+1)\alpha} \, U_{\ell\alpha}(x) - \varepsilon \, U_{k\alpha}(x) - x^{(k+1)\alpha} \, = 0 \,, \quad U_{1\alpha} (x) = x^\alpha - \varepsilon \,, \quad k = 1, 2, ... \,,\label{3.10} \\
&\sum_{\ell =1}^{k+1} \left(\begin{aligned}k&+1\\ &\ell\end{aligned}\right) \, x^{(k-\ell+1)\alpha} \, V_{(\ell-1)\alpha}(x) - \varepsilon \, V_{(k-1)\alpha}(x) + \varepsilon \,
\nu!  \, \nu^k \, x^{k\alpha + \beta} \, = 0 \,, \quad V_0 (x) = - \varepsilon \, \nu! \, x^\beta \,, \quad k = 1, 2, ... \,,  \label{3.11}\\
&\sum_{\ell =1}^{k+1} \left(\begin{aligned} k&+1\\ &\ell\end{aligned}\right) \, x^{(k-\ell+1)\alpha} \, A_{(\ell-1)\alpha}(N;x) - \varepsilon \, A_{(k-1)\alpha}(N;x) - (N +\nu)^k \, x^{k\alpha} \, = 0 \,, \quad  A_0 (N; x) = 1 \,, \quad k = 1, 2, ... \,  . \label{3.12}
\end{align}
\end{theorem}

\begin{proof}
Formula \eqref{3.9} can be rewritten as
\begin{align}
S_k (N; x) = - x^{- k\alpha} U_{k\alpha} (x) S_0 (N; x) + x^{- k\alpha} V_{(k-1)\alpha} (x) + A_{(k-1)\alpha} (N; x) x^{- k\alpha} \varepsilon^{N-1} (N +\nu )! x^{\alpha N + \beta} . \label{3.13}
\end{align}
Now one can replace $S_k (N; x)$ in recurrence relation \eqref{3.4} by this one in \eqref{3.13}. Compiling  the terms separately  with $S_0 (n; s) ,$ then with $\, x^{\alpha N + \beta}$ and finally all the rest terms, we obtain respectively recurrence relations for $U_{k\alpha} (x) , \, A_{(k-1)\alpha} (N; x)$ and $V_{(k-1)\alpha} (x) .$
\end{proof}

Note that factor $x^\beta$ does not play an important role in \eqref{3.9}, because $V_{(k-1)\alpha} (x)$ also contains $x^\beta$ and it can be excluded from this formula by redefinition of  $V_{(k-1)\alpha} (x)$.

\begin{theorem}  \label{Th2}
Polynomials  $U_{k\alpha} (x) $ and $V_{(k-1)\alpha} (x)$ are related to polynomial $ A_{(k-1)\alpha} (N; x)$ in the  form
\begin{align}
& U_{k\alpha} (x) = (\nu + 1) x^\alpha A_{(k-1)\alpha} (1; x) - \varepsilon A_{(k-1)\alpha} (0; x) - \nu^k x^{k\alpha} \,, \quad  k \in \mathbb{N} \,,  \label{3.14} \\
& V_{(k-1)\alpha} (x) = - \varepsilon \, \nu! \, x^\beta \, A_{(k-1)\alpha} (0; x) \,, \quad  k \in \mathbb{N} .   \label{3.15}
\end{align}
\end{theorem}

\begin{proof}
We  use equation \eqref{3.9}.  Note that $U_{k\alpha} (x) $ and $V_{(k-1)\alpha} (x)$ do not depend on the upper limit of summation in \eqref{3.9}. Hence,  subtracting
equations in  \eqref{3.9} with $\sum_{n=0}^{N-1}$ and $\sum_{n=0}^{N-2} ,$ we obtain relation
\begin{equation}
(N + \nu -1)^k \, x^{k\alpha} + U_{k\alpha} (x) = (N + \nu) \, x^\alpha \, A_{(k-1)\alpha} (N; x) - \varepsilon \, A_{(k-1)\alpha} (N-1; x)  \label{3.16}
\end{equation}
which does not contain $V_{(k-1)\alpha} (x) .$  Taking $N = 1$ in \eqref{3.16}, one obtains expression \eqref{3.14} for $U_{k\alpha} (x) .$ Now using \eqref{3.9}
when $N= 1$ gives
\begin{equation}
\nu! \, \nu^k \, x^{k\alpha + \beta} + U_{k\alpha} (x) \, x^\beta \, \nu! = V_{(k-1)\alpha} (x) + A_{(k-1)\alpha} (1; x) \, (\nu + 1)! \, x^{\alpha + \beta} . \label{3.17}
\end{equation}
Combining \eqref{3.16} and  \eqref{3.17}, it follows \eqref{3.15}.
\end{proof}

Recurrent formulas \eqref{3.10}--\eqref{3.12} enable to calculate  polynomials $U_{k\alpha} (x), \, V_{(k-1)\alpha} (x)$ and $A_{(k-1)\alpha} (N; x)$ for any $k \in \mathbb{N}$,
knowing initial expressions: $U_1(x) = x^\alpha - \varepsilon, \, V_0(x) =
- \varepsilon \, \nu! \, x^\beta$ and $A_0 (N; x) = 1$.
For the first five values of degree $k$, we have obtained the following explicit expressions.

\begin{itemize}
\item  $k = 1$
\begin{align}
U_{1\alpha}(x) =& x^\alpha -\varepsilon, \nonumber\\
V_0(x) =& - \varepsilon \, \nu! \, x^\beta, \nonumber \\
A_0 (n; x) =& 1. % \label{3.14}
\end{align}
\item  $k = 2$
\begin{align}
U_{2\alpha}(x) =& - x^{2\alpha} + 3 \varepsilon x^\alpha -1, \nonumber\\
V_{1\alpha}(x) =& -\varepsilon \nu! x^\beta [(\nu - 2) x^\alpha +  \varepsilon], \nonumber\\
A_{1\alpha}(n; x) =& (n + \nu -2) x^\alpha  + \varepsilon .  % \label{3.15}
\end{align}
\item $k = 3$
\begin{align}
U_{3\alpha}(x) =&  x^{3\alpha} - 7 \varepsilon x^{2\alpha}  + 6 x^\alpha -\varepsilon, \nonumber\\
V_{2\alpha}(x) =& -\varepsilon \nu! x^\beta [(\nu^2 - 3 \nu + 3) x^{2\alpha} + (\nu -5) \varepsilon x^\alpha + 1], \nonumber\\
A_{2\alpha}(n; x) =& [(n+\nu)^2 -3(n +\nu) +3] x^{2\alpha} + (n + \nu -5)\varepsilon x^\alpha  + 1 .  %\label{3.16}
\end{align}
\item   $k = 4$
\begin{align}
U_{4\alpha}(x) =&  - x^{4\alpha} + [\nu^3 (1- \varepsilon) - 4 \nu^2 (1 - \varepsilon) + 6 \nu (1 - \varepsilon) + 11 + 4 \varepsilon] x^{3\alpha}
\nonumber \\ &+ [\nu^2 (1- \varepsilon) - 7 \nu (1- \varepsilon) - 8 - 17 \varepsilon] x^{2\alpha} + 10 \varepsilon x^\alpha - 1  , \nonumber\\
V_{3\alpha}(x) =& -\varepsilon \nu! x^\beta [(\nu^3 - 4 \nu^2 + 6 \nu - 4) x^{3\alpha} + (\nu^2 - 7 \nu + 17) \varepsilon x^{2\alpha} + (\nu -9) x^\alpha + \varepsilon], \nonumber\\
A_{3\alpha}(n; x) =& [(n+ \nu )^3 - 4 (n + \nu)^2 +6 (n + \nu) -4] x^{3\alpha} + [(n + \nu )^2 -7 (n + \nu) +17] \varepsilon x^{2\alpha}   \nonumber\\
 &+ (n + \nu -9) x^\alpha  + \varepsilon . % \label{3.17}
\end{align}
\item    $k = 5$
\begin{align}
U_{5\alpha}(x) =& x^{5\alpha} - (\nu^3 + 31) \varepsilon x^{4\alpha} + 90 x^{3\alpha}  -65 \varepsilon x^{2\alpha} + 15 x^\alpha -\varepsilon ,  \nonumber \\
V_{4\alpha}(x) =& -\varepsilon \nu! x^\beta [(\nu^4 - 5 \nu^3 + 10 \nu^2 - 10 \nu + 5) x^{4\alpha} + (\nu^3 - 9\nu^2 + 31 \nu - 49) \varepsilon x^{3\alpha} \nonumber \\ &+ (\nu^2 - 12 \nu + 52) x^{2\alpha} + (\nu - 14) \varepsilon x^\alpha + 1],   \nonumber\\
A_{4\alpha}(n; x) =&  [(n + \nu )^4 -5 (n + \nu )^3 +10 (n + \nu)^2 -10 (n + \nu) +5] x^{4\alpha}  \nonumber \\ &+  [(n + \nu )^3 - 9 (n + \nu )^2 + 31
(n + \nu ) - 49] \varepsilon x^{3\alpha}  \nonumber \\ &+ [(n + \nu )^2 -12 (n + \nu ) + 52] x^{2\alpha} + (n + \nu - 14) \varepsilon  x^\alpha + 1 . % \label{3.18}
\end{align}
%\item  $k = 6$
%\begin{align}
%&U_6(x) = -x^6 + 63x^5 -301 x^4 +350 x^3 - 140 x^2 +21 x -1, \nonumber \\
%&V_6(x) = 6 x^5 -129x^4+246x^3 - 121 x^2 + 20 x -1,   \nonumber  \\
%&A_5(n; x) = (n^5 -6 n^4 +15 n^3 - 20 n^2 +15 n -6) x^5 + (n^4 - 11n^3  \nonumber \\ & \qquad \qquad
%+ 49n^2 - 111n  +129)x^4 + (n^3 - 15n^2 + 88 n - 246) x^3 \nonumber \\ & \qquad \qquad + (n^2 - 18 n +121) x^2 + (n - 20) x + 1 .  % \label{3.18a}
%\end{align}
\end{itemize}

It is worth emphasizing that all the above equalities, in particular \eqref{3.4} and \eqref{3.9}, are valid in real and all $p$-adic cases. The central role in \eqref{3.9} plays
polynomial $A_{k\alpha} (N; x), $ which is solution of the recurrence relation \eqref{3.12}, because  polynomials $U_{k\alpha} (x)$ and $V_{(k-1)}(x)$ are simply connected to $A_{k\alpha} (N; x) $ by formulas \eqref{3.14} and \eqref{3.15}, respectively.
 When $N \to \infty$ in \eqref{3.9}, the term with polynomial
$A_{(k-1)\alpha} (N; x)$ $p$-adically vanishes giving the sum of the following $p$-adic infinite functional series:
\begin{equation}
 \sum_{n=0}^{\infty} \varepsilon^n (n + \nu)! \, [ (n + \nu)^k \, x^{k\alpha} + U_{k\alpha} (x) ] \, x^{\alpha n+\beta} = V_{(k-1)\alpha} (x) .    \label{3.19}
\end{equation}
This equality has the same form for any $k \in \mathbb{N} ,$ and polynomials $U_{k\alpha} (x)$ and $V_{(k-1)\alpha} (x)$ separately have the same values in all $p$-adic cases for any
$x \in \mathbb{Z} .$ In other words, nothing depends on particular $p$-adic properties in \eqref{3.19} when $x \in \mathbb{Z}$, i.e. this is $p$-adic invariant
result. This result gives us the possibility to present a general solution of the problem posed on $p$-adic invariant summation of the series \eqref{2.1}.

\begin{theorem}   \label{Th3}
The functional series \eqref{2.1}  has $p$-adic invariant sum
\begin{equation}
S_{k\alpha} (x) \equiv \sum_{n = 0}^{+\infty} \varepsilon^n \, (n + \nu)! \, P_{k\alpha} (n; x) \, x^{\alpha n + \beta} = Q_{k\alpha} (x)   \label{3.20}
\end{equation}
if
\begin{equation}
P_{k\alpha} (n; x) = \sum_{j =1}^k B_j \, [(n+\nu)^{j} x^{j\alpha} + U_{j\alpha} (x)]    \quad \text{and} \quad   Q_{k\alpha} (x) = \sum_{j =1}^k B_j \, U_{j\alpha} (x),  \label{3.21}
\end{equation}
where $B_j, \, x \in \mathbb{Z} .$
\end{theorem}

Note that
$A_{k\alpha} (n; x)$ as well as $U_{k\alpha} (x)$ and $V_{(k-1)\alpha} (x)$ can be written in the compact form
\begin{equation}
A_{k\alpha} (n; x) = \sum_{\ell = 0}^k A_{(k\alpha)\ell} (n+ \nu) \, x^{\ell\alpha} , \quad U_{k\alpha} (x)= \sum_{\ell = 0}^k  U_{(k\alpha)\ell} \, x^{\ell\alpha} ,   \quad V_{k\alpha} (x) = \sum_{\ell = 0}^k  V_{(k\alpha)\ell} \, x^{\ell\alpha} ,      \label{3.22}
\end{equation}
where $A_{(k\alpha)\ell} (n+\nu)$ is a polynomial in $n +\nu$ of degree $\ell$ with $(n +\nu)^\ell$ as the term of the highest degree.

Putting $x = 0$ in \eqref{3.10}--\eqref{3.12}, the following properties hold:
\begin{itemize}
\item  $A_{k\alpha}(n; 0) = \varepsilon A_{(k-1)\alpha}(n; 0) = \varepsilon^k ,  \, \, \, k =  1, 2, ...$
\item  $ U_{(k+1)\alpha} (0) = \varepsilon U_{k\alpha} (0) = - \varepsilon^{k+1} , \, \, \,  k =  1, 2, ...$
\item  $ V_{k\alpha} (0) = \varepsilon V_{(k-1)\alpha} (0) = -\nu! x^\beta \varepsilon^{k+1} , \, \, \, k =  1, 2, ...$
\end{itemize}

As an illustration of  summation formula \eqref{3.19}, we present five simple (k = 1, ...,5) examples.

\begin{itemize}
\item  $ k = 1$
\begin{align}
 \sum_{n=0}^{\infty} \varepsilon^n \,(n+\nu)! \, [(n + \nu + 1)x^\alpha  - \varepsilon]\, x^{\alpha n} = - \varepsilon\, \nu! \,,  \quad x \in \mathbb{Z} .   \label{3.23}
\end{align}

\item $k = 2$
\begin{align}
 \sum_{n=0}^{\infty} \varepsilon^n \,(n+\nu)! \, \{[(n+\nu)^2 - 1] x^{2\alpha} + 3 \varepsilon x^\alpha - 1\}\, x^{\alpha n} = \varepsilon\, \nu!\, [(2 - \nu)x^\alpha -\varepsilon] \,,  \quad x \in \mathbb{Z} .  \label{3.24}
\end{align}

\item $k = 3$
\begin{align}
 &\sum_{n=0}^{\infty} \varepsilon^n \,(n+\nu)! \, \{[(n+\nu)^3 + 1] x^{3\alpha}  - 7 \varepsilon  x^{2\alpha} + 6 x^\alpha - \varepsilon\} \, x^{\alpha n} \nonumber \\ &= - \varepsilon\, \nu! [(\nu^2 -3 \nu +3) x^{2\alpha} + (\nu -5) \varepsilon x^\alpha +1] ,  \quad x \in \mathbb{Z} . \label{3.25}
\end{align}

\item $k = 4$
\begin{align}
 &\sum_{n=0}^{\infty} \varepsilon^n \,(n+\nu)! \, \{ \,[(n+\nu)^4 - 1] x^{4\alpha} + [\nu^3 (1- \varepsilon) - 4 \nu^2 (1 - \varepsilon) + 6 \nu (1 - \varepsilon) + 11 + 4 \varepsilon] x^{3\alpha} \nonumber \\
&+ [\nu^2 (1- \varepsilon) - 7 \nu (1- \varepsilon) - 8 - 17 \varepsilon] x^{2\alpha} + 10 \varepsilon x^\alpha - 1\,  \}\, x^{\alpha n} \nonumber \\
&= -\varepsilon \, \nu! \, [(\nu^3 - 4 \nu^2 + 6 \nu - 4) x^{3\alpha} + (\nu^2 - 7 \nu + 17) \varepsilon x^{2\alpha} + (\nu -9) x^\alpha + \varepsilon] ,  \quad x \in \mathbb{Z} . \label{3.25}
\end{align}

\item $k = 5$
\begin{align}
 &\sum_{n=0}^{\infty} \varepsilon^n \,(n+\nu)! \, \{ \,[(n+\nu)^5 + 1] x^{5\alpha} - (\nu^3 + 31) \varepsilon x^{4\alpha} + 90 x^{3\alpha}  -65 \varepsilon x^{2\alpha} + 15 x^\alpha -\varepsilon  \, \} \, x^{\alpha n} \nonumber \\
 &= -\varepsilon \, \nu! \, [(\nu^4 - 5 \nu^3 + 10 \nu^2 - 10 \nu + 5) x^{4\alpha} + (\nu^3 - 9\nu^2 + 31 \nu - 49) \varepsilon x^{3\alpha} \nonumber \\ &+ (\nu^2 - 12 \nu + 52) x^{2\alpha} + (\nu - 14) \varepsilon x^\alpha + 1]  ,  \quad x \in \mathbb{Z} . \label{3.25}
\end{align}

\end{itemize}

\section{Discussion and Concluding Remarks}

The main results presented in this paper are summation formula \eqref{3.4} and theorems (\ref{Th1})--(\ref{Th3}). These results are  generalizations of some earlier results, see  \cite{bd9,bd10,bd11,bd14}.

Finite series \eqref{3.4} with their sums \eqref{3.9}   are valid for real and  $p$-adic numbers. When $n \to \infty$ the corresponding infinite
series are divergent in real case, but are convergent and have  the same sums in all $p$-adic cases. This fact can be used to extend
these sums to the real case. Namely, the sum of a divergent series depends on the way of its summation and here it can be used its integer sum                valid in all $p$-adic number fields. This way of summation of real divergent series was introduced for the first time in
\cite{bd2} and called adelic summability. An importance of this adelic summability  depends on its potential future use in some concrete examples.

The simplest infinite series with $n!$ is $\sum n! .$ It is convergent in all $\mathbb{Z}_p ,$ but has not $p$-adic invariant sum. Even it is not
known so far does it has a rational sum in any $\mathbb{Z}_p .$  Rationality of this series and $\sum n! n^k x^n$ was discussed in \cite{bd9}. The series
$\sum n!$ is also related to Kurepa hypothesis which  states $(!n, n!) = 2, \quad 2 \leq n \in \mathbb{N} ,$  where $!n = \sum_{j=0}^{n-1} j!$. Validity
of this hypothesis is still an open problem in number theory. There are many equivalent statements to the Kurepa hypothesis, see  \cite{bd10} and references
therein. From $p$-adic point of view, the Kurepa hypothesis reads: $\sum_{j=0}^{\infty} j!  = n_0 + n_1 p + n_2 p^2 + \cdots ,$ where  digit $n_0 \neq 0$ for all primes $p \neq 2 .$

It is worth emphasizing that polynomials $A_{k\alpha}(n; x)$ contain all information about properties of series \eqref{3.9}. For various combinations of $x = 0, \pm 1, \pm 2, ..., $ \, \, $n = 0, 1, 2, ...$ and parameters $k, \alpha, \in \mathbb{N}$ one can obtain integer sequences, and some of them are already known. Note that parameter $\nu$ in polynomials $A_{k\alpha}(n; x)$ appears in the form $n + \nu$ and it is enough to consider how $A_{k\alpha}(n; x)$ depends on $n$. Hence we will consider $A_{k\alpha}(n; x)$ with  parameter $\nu = 0 .$  Here are some simple sequences derived from $A_{k\alpha}(n; x) .$
\begin{itemize}
\item $\alpha =$ any even natural number :
\begin{align}
&A_{k\alpha}(0; \pm 1): \quad 1, \, \varepsilon - 2, \, 4 - 5 \varepsilon , \, -13 + 18 \varepsilon, \,  58 - 63 \varepsilon ,  ...  \qquad k \in \mathbb{N}_0 \,, \quad \varepsilon = \pm 1 \,, \nonumber \\
&A_{k\alpha}(1; \pm 1): \quad 1, \, -1 + \varepsilon , \, 2 - 4 \varepsilon , \, -9 + 12 \varepsilon, \,  43 - 39 \varepsilon , \, ...  \qquad k \in \mathbb{N}_0 \,,  \quad \varepsilon = \pm 1 \,. \nonumber
\end{align}

\item $\alpha =$ any odd natural number :
\begin{align}
&A_{k\alpha}(0;  1): \quad 1, \, \, \varepsilon - 2, \,  4 - 5 \varepsilon , \, \, -13 + 18 \varepsilon, \, \, 58 - 63 \varepsilon , \, \, ...  \qquad k \in \mathbb{N}_0 \,, \quad \varepsilon = \pm 1 \,, \nonumber \\
&A_{k\alpha}(1;  1): \quad 1, \, \, -1 + \varepsilon , \, \, 2 - 4 \varepsilon , \, \, -9 + 12 \varepsilon, \, \,  43 - 39 \varepsilon , \, \, ...  \qquad k \in \mathbb{N}_0 \,,  \quad \varepsilon = \pm 1 \,. \nonumber
\end{align}

\item $\alpha =$ any odd natural number :
\begin{align}
&A_{k\alpha}(0; - 1): \quad 1, \, \, 2 + \varepsilon , \, \, 4 + 5 \varepsilon , \, \, 13 + 18 \varepsilon, \, \,  58 + 63 \varepsilon , \, \, ...  \qquad k \in \mathbb{N}_0 \,,  \quad \varepsilon = \pm 1 \,, \nonumber \\
&A_{k\alpha}(1; - 1): \quad 1, \, \, 1 + \varepsilon , \, \,  4 \varepsilon , \, \, 9 + 12 \varepsilon, \, \,  43 + 39 \varepsilon , \, \, ... \qquad \, \qquad k \in \mathbb{N}_0 \,, \quad \varepsilon = \pm 1 \,. \nonumber
\end{align}

\end{itemize}

Below are also some simple integer sequences derived from  $V_{k\alpha} (x)$ and $U_{k\alpha} (x) .$

\begin{itemize}
\item $ x = 1 , \quad \alpha \in \mathbb{N} , \quad \beta = 0 , \quad \nu = 0 \, :$
\begin{align}
& V_{k\alpha} (1) :  \quad - \varepsilon , \, \, - 1 + 2 \varepsilon , \, \,  5 - 4 \varepsilon , \, \, - 18 + 13 \varepsilon , \, \, 63 - 58 \varepsilon , \, \, \, ...  \qquad  \qquad k \in \mathbb{N}_0 \,, \quad \varepsilon = \pm 1 \,, \nonumber \\
& U_{k\alpha} (1) : \quad 1 - \varepsilon , \, \, - 2 + 3 \varepsilon , \, \,  7 - 8 \varepsilon , \, \,  1 - 3 \varepsilon , \, \, 106 - 87 \varepsilon , \, \, \, ... \qquad \, \qquad k \in \mathbb{N} \,, \quad \varepsilon = \pm 1 \, \nonumber .
\end{align}
\item  $ x = 1 , \quad \alpha \in \mathbb{N} , \quad \beta = 0 , \quad \nu = 1 \, :$
\begin{align}
& V_{k\alpha} (1) :   \quad - \varepsilon , \, \, - 1 + 2 \varepsilon , \, \,  5 - 4 \varepsilon , \, \, - 18 + 13 \varepsilon , \, \,   63 - 58 \varepsilon , \, \, \, ... \qquad  \qquad k \in \mathbb{N}_0 \,, \quad \varepsilon = \pm 1 \,, \nonumber \\
& U_{k\alpha} (1) :  \quad  1 - \varepsilon , \, \, - 2 + 3 \varepsilon , \, \,  7 - 8 \varepsilon , \, \, - 14 + 3 \varepsilon , \, \, - 2 , \, \, \, ... \qquad  \qquad \, \,\, \quad k \in \mathbb{N} \,, \quad \varepsilon = \pm 1 \, \nonumber .
\end{align}
\end{itemize}

When $x = \pm 1, \,\, \varepsilon = \alpha = 1, \,\, \beta = \nu = 0$, then \eqref{3.19} becomes
\begin{equation}
 \sum_{n=0}^{\infty} n! \, [ n^k  + u_k ]  = v_k  \quad \text{if} \, \, \, x =1 ,  \qquad \qquad  \sum_{n=0}^{\infty} (-1)^n n! \, [ (-1)^{k+1} n^k  + \bar{u}_k ]  = \bar{v}_k  \quad \text{if} \, \, \, x = - 1 ,  \label{3.26}
\end{equation}
where $u_k = U_{k1} (1), \, \, v_k = V_{(k-1)1} (1)$ and $\bar{u}_k = -U_{k1} (-1), \, \, \bar{v}_k = - V_{(k-1)1} (-1)$ are some integers. First equality in \eqref{3.26} was introduced in \cite{bd5}, and properties of $u_k$ and $v_k$ are investigated in series of papers by Dragovich (see references \cite{bd8,bd9,bd10,bd11}). In \cite{murty} some relationships  of $u_k$ with the Stirling  numbers of the second kind are established, and $p$-adic irrationality of $\sum_{n \geq 0} n! n^k$ was discussed (see \cite{subedi1,subedi2,alexander}). Note that the following sequences are related to some real (combinatorial) cases, compare with \cite{sloane}:
\begin{align}
&v_k =  -A_{(k-1)1} (0;1) = V_{(k-1)1} (1) : \, \, -1, 1, 1, -5, 5, 21, - 105, 141,  ...        \quad  \text{see  \, A014619}  \label{3.27}\\
&u_k = A_{(k-1)1} (1;1) - A_{(k-1)1}(0;1) = U_{k1} (1) : \, \, 0, 1, -1, -2, 9, -9, -50, 267, ...  \quad  \text{see  \, A000587} \label{3.28} \\
&\bar{u}_k = A_{(k-1)1} (1;-1) + A_{(k-1)1}(0;-1) = -U_{(k1} (-1): \, \, 2, 5, 15, 52, 203, 877, 4140, 21147, ... \, \, \,  \text{see \,   A000110} \label{3.29} \\
&\bar{v}_k = A_{(k-1)1} (0;-1) = -V_{(k-1)1} (-1) : \, \, 1, 3, 9, 31, 121, 523, 2469, 12611, ...   \quad  \text{see  \, A040027}   \label{3.30} .
\end{align}

It is worth pointing out  integer series \eqref{3.29} and \eqref{3.30}, which are directly calculated from the following recurrence relations:
\begin{align}
& \bar{u}_{k+1} = \sum_{\ell =1}^{k} \left(\begin{aligned}k&+1\\ &\ell\end{aligned}\right) \, (-1)^{k-\ell} \, \bar{u}_{\ell} + \bar{u}_{k} + (-1)^k  \,,
 \qquad   \bar{u}_{1} = 2 \,,  \quad  k = 1, 2, 3, ...  \,,     \label{3.31}    \\
& \bar{v}_{k+1} =  \sum_{\ell =1}^{k} \left(\begin{aligned}k&+1\\ &\ell\end{aligned}\right) \, (-1)^{k-\ell} \, \bar{v}_{\ell} + \bar{v}_{k} \,, \qquad \bar{v}_{1}  = 1 \,, \quad  k = 1, 2, 3, ... \, \, .  \label{3.32}
\end{align}
In particular, the series of integers $\bar{u}_k \,, \, \, (k = 1, 2, 3, ...) \, $ coincides with the Bell numbers $B_k \,, \, \, (k = 0, 1, 2, ...) \, $ by equality
$B_{k+1} = \bar{u}_{k}$ for $k \geq 1$ (at least for the first 8 terms directly calculated). Recall that the Bell numbers $B_k$ are equal to the number of partitions of a set of $k$ elements. They satisfy the recurrence relation   $$B_{k+1} = \sum_{\ell =0}^{k} \left(\begin{aligned}&k\\ &\ell\end{aligned}\right)  \, B_{\ell} \,, \qquad B_0 = 1 .$$   It follows that the number of partitions of sets with  more than one element can be obtained also from the recurrence relation for $\bar{u}_{k}$ given by \eqref{3.31}.

Various  aspects of the  polynomials $A_{k\alpha} (n; x) \,, \, \,  (k = 0, 1, 2, ... \,, \, \, \alpha = 1, 2, 3, ...) \, $ deserve to be further  analyzed.

\section*{Acknowledgments}
Work on this paper was partially supported by Ministry of Education, Science and Technological Development of the Republic of Serbia, projects: OI 174012, TR 32040 and TR 35023. A part of this work was done during B.D. visit  of the
International Center for Mathematical Modeling
in Physics, Engineering, Economics, and Cognitive Science,
Linnaeus University, V\"axj\"o, Sweden.

%%% ENTER REFERENCES IN THE FORM


\begin{thebibliography}{33}



\bibitem{bd1} I. Ya. Arefeva, B. Dragovich and I. V. Volovich,  ``On the $p$-adic summability of the anharmonic oscillator,''Phys. Lett. B
      \textbf{200}, 512--514 (1988).

\bibitem{bd2} B. Dragovich, ``$p$-Adic perturbation series and adelic summability,'' Phys. Lett. B \textbf{256} (3,4), 392--39 (1991).

\bibitem{bd3} B. G. Dragovich, ``Power series everywhere convergent on $R$ and $Q_p$,'' J. Math. Phys.  \textbf{34} (3), 1143--1148 (1992)
[arXiv:math-ph/0402037].

\bibitem{bd4} B. G. Dragovich, ``On $p$-adic aspects of some perturbation series,'' Theor. Math. Phys. \textbf{93} (2), 1225--1231 (1993).

\bibitem{bd5} B. G. Dragovich, ``Rational summation of $p$-adic series,''  Theor. Math. Phys. \textbf{100} (3), 1055--1064 (1994).

\bibitem{bd6} B. Dragovich, ``On $p$-adic series in mathematical physics,'' Proc. Steklov  Inst. Math. \textbf{203}, 255--270 (1994).

\bibitem{bd7} B. Dragovich, ``On $p$-adic series with rational sums,'' Scientific Review \textbf{19--20}, 97--104 (1996).

\bibitem{bd8} B. Dragovich, ``On some $p$-adic series with factorials,'' in \textit{$p$-Adic Functional Analysis}, Lect. Notes Pure
Appl. Math. \textbf{192}, 95--105 (Marcel Dekker, 1997)  [arXiv:math-ph/0402050].

\bibitem{bd9} B. Dragovich, ``On $p$-adic power series,''  in \textit{$p$-Adic Functional Analysis}, Lect. Notes Pure
Appl. Math. \textbf{207}, 65--75 (Marcel Dekker, 1999) [arXiv:math-ph/0402051].

\bibitem{bd10} B. Dragovich, ``On some finite sums with factorials,'' Facta Universitatis: Ser. Math. Inform. \textbf{14}, 1--10 (1999) 	
[arXiv:math/0404487 [math.NT]].

\bibitem{bd11} M. de Gosson, B. Dragovich and A. Khrennikov, ``Some $p$-adic differential equations,''  in \textit{$p$-Adic Functional Analysis}, Lect. Notes Pure
Appl. Math. \textbf{222}, 91--112 (Marcel Dekker, 2001) [arXiv:math-ph/0010023].

\bibitem{freund} L.~Brekke and P.~G.~O.~Freund, ``$p$-adic numbers in physics,'' Phys. Rep. \textbf{233}, 1--66 (1993).

\bibitem{vvz} V.~S.~Vladimirov, I.~V.~Volovich and E.~I.~Zelenov, \textit{$p$-Adic Analysis and Mathematical Physics} (World Sci. Publ., Singapore, 1994).

\bibitem{bd12} B. Dragovich, A. Yu. Khrennikov, S. V. Kozyrev and I. V. Volovich, ``On $p$-adic mathematical physics,'' $p$-Adic Numbers Ultram. Anal.
Appl. \textbf{1} (1), 1--17 (2009) 	[arXiv:0904.4205 [math-ph]].

\bibitem{bd13} B. Dragovich and A. Yu. Dragovich, ``A $p$-adic model of DNA sequence and genetic code,'' $p$-Adic Numbers Ultram. Anal.
Appl. \textbf{1} (1), 34--41 (2009) [arXiv:q-bio/0607018 [q-bio.GN]].

\bibitem{AK1} S. Albeverio, R. Cianci and A. Yu.  Khrennikov, ``$p$-Adic valued
quantization,'' $p$-Adic Numbers, Ultrametric Analysis and
Applications {\bf 1} (2), 91--104 (2009).

\bibitem{AK2} S. Albeverio,  A.  Khrennikov and R. Cianci, ``On the spectrum of
the $p$-adic position operator,'' J. Physics A: Math. and
General {\bf 30}, 881--889 (1997).

\bibitem{AK3} S. Albeverio,  A.  Khrennikov and R. Cianci, ``On the Fourier
transform and the spectral properties of the $p$-adic momentum and
Schrodinger operators,''  J. Physics A: Math. and General {\bf
30}, 5767--5784 (1997).

\bibitem{AK4} S. Albeverio,  A.  Khrennikov and R. Cianci, ``A representation
of quantum field hamiltonian in a $p$-adic Hilbert space,''
Theor. Math. Physics {\bf 112} (3),  355--374 (1997).

\bibitem{schikhof} W. H. Schikhof, \textit{Ultrametric Calculus: An Introduction to $p$-Adic Analysis} (Cambridge Univ. Press, Cambridge, 1984).

\bibitem{bd14} B. Dragovich and N. Z. Misic, ``$p$-Adic invariant summation of some $p$-adic functional series,'' $p$-Adic Numbers Ultr. Anal. Appl.
\textbf{6}  (4), 275--283 (2014), [arXiv:1411.4195v1 [math.NT]].

\bibitem{murty} M. Ram Murty and S. Sumner, ``On the $p$-adic series $\sum_{n=1}^\infty n^k \cdot n! ,$''  in \textit{Number Theory}, CRM Proc.
Lecture Notes \textbf{36}, 219--227 (Amer. Math. Soc., 2004).

\bibitem{subedi1} P.~K.~Saikia and D.~Subedi, ``Bell numbers, determinants and series,'' Proc. Indian Acad. Sci. (Math. Sci.) \textbf{123} (2),
151--166 (2013).

\bibitem{subedi2} D.~Subedi, ``Complementary Bell numbers and $p$-adic series,'' J. Integer Seq. \textbf{17}, 1--14 (2014).

\bibitem{alexander}  N.~C.~Alexander, ``Non-vanishing of Uppuluri-Carpenter numbers,''  Preprint  http://tinyurl.com/oo36das.

\bibitem{sloane}N.~J.~A.~Sloane, ``The on-line encyclopedia of integer sequences,'' https://oeis.org/.

\end{thebibliography}
\end{document}